\documentclass{amsart}

\usepackage[letterpaper]{geometry} %% Letter setup

\usepackage{etex} %% This is for an extension and allows us to use many packages.
\usepackage{amsmath} % Necessary
\usepackage{amssymb} % Necessary
\usepackage{amsthm} % Necessary
\usepackage{amsfonts} % Necessary
\usepackage{booktabs} % For tables that look nice with the command \begin{tabular}
\usepackage[usenames,dvipsnames,svgnames]{xcolor} % For tonnes of colours and control over creating colours
\usepackage{hyperref} % Automatically turns all internal references to links
\hypersetup{colorlinks=true,
allcolors=blue!75!black} % Changes links to be dark blue with no annoying boxes around them
\usepackage{mathrsfs} % Allows the use of \mathscr letters
\usepackage{verbatim} % Has a comment environment for commenting out large chunks of code
\usepackage{tikz}
\usepackage{tikz-cd}
\usepackage{caption}
\usetikzlibrary{decorations.markings}

\newcommand{\NN}{\mathbb{N}}

\newcommand{\ZZ}{\mathbb{Z}}

\newcommand{\vC}{\mathcal{C}}

\newcommand{\vE}{\mathcal{E}}

\newcommand{\vH}{\mathcal{H}}
\newcommand{\vI}{\mathcal{I}}

\newcommand{\vS}{\mathcal{S}}

\newcommand{\fI}{\mathfrak{I}}

\newcommand{\fP}{\mathfrak{P}}

\newcommand{\im}{\operatorname{im}}

\newcommand{\Mod}{\operatorname{Mod}}

\newcommand{\Ob}{\operatorname{Ob}}

\newcommand{\sgn}{\operatorname{sgn}}

\newcommand{\id}{\text{id}}
\newcommand{\lra}{\longrightarrow}

\newcommand{\ol}[1]{\overline{#1}}

\newcommand{\sm}{\setminus}
\newcommand{\wt}[1]{\widetilde{#1}}

\newcommand{\eand}{\quad \text{ and } \quad}

\definecolor{lightgrey}{gray}{.85}

\theoremstyle{plain}
\newtheorem{thm}{Theorem}[section]
\newtheorem{lem}[thm]{Lemma}
\newtheorem{cor}[thm]{Corollary}
\newtheorem{prop}[thm]{Proposition}

\newtheorem{const}[thm]{Construction}
\newtheorem{quest}[thm]{Question}

\theoremstyle{definition}
\newtheorem{defn}[thm]{Definition}
\newtheorem{rmk}[thm]{Remark}

\newtheorem{eg}[thm]{Example}

\newtheorem{problem}[thm]{Problem}

\renewcommand{\emph}{\bf}

\title{circularly ordering direct products and the obstruction to left-orderability}
\author[Adam Clay]{Adam Clay}
\thanks{Adam Clay was partially supported by NSERC grants RGPIN-2014-05465 and RGPIN-05343-2020}
\address{Department of Mathematics\\
University of Manitoba \\
Winnipeg \\
MB Canada R3T 2N2} \email{Adam.Clay@umanitoba.ca}
\urladdr{http://server.math.umanitoba.ca/~claya/} 

\author[Tyrone Ghaswala]{Tyrone Ghaswala}
\thanks{Tyrone Ghaswala was partially supported by a PIMS postdoctoral fellowship at the University of Manitoba and a CIRGET postdoctoral fellowship at L'Universit\'{e} du Qu\'{e}bec \`{a} Montr\'{e}al}
\address{D\'{e}partement de Math\'{e}matiques\\
Universit\'{e} du Qu\'{e}bec \`{a} Montr\'{e}al \\
Montr\'{e}al \\
QC Canada H2X 3Y7} \email{ty.ghaswala@gmail.com}
\urladdr{https://server.math.umanitoba.ca/~ghaswalt/} 
\date{\today}

\begin{document}

\maketitle

\begin{abstract}
Motivated by the recent result that left-orderability of a group $G$ is intimately connected to circular orderability of direct products $G \times \ZZ/n\ZZ$, we provide necessary and sufficient cohomological conditions that such a direct product be circularly orderable. As a consequence of the main theorem, we arrive at a new characterization for the fundamental group of a rational homology 3-sphere to be left-orderable. Our results imply that for mapping class groups of once-punctured surfaces, and other groups whose actions on $S^1$ are cohomologically rigid, the products $G \times \ZZ/n\ZZ$ are seldom circularly orderable. We also address circular orderability of direct products in general, dealing with the cases of factor groups admitting a bi-invariant circular ordering, and iterated direct products whose factor groups are amenable.
\end{abstract}

\section{Introduction}

A group $G$ is called {\it left-orderable} if it admits a strict total ordering $<$ of its elements such that $g<h$ implies $fg <fh$ for all $f, g, h \in G$.   When $G$ is countable, it is left-orderable if and only if there exists an embedding of $G$ into $\mathrm{Homeo}_+(\mathbb{R})$. Related to this idea, a group $G$ is {\it circularly-orderable} if and only if it admits an inhomogeneous $2$-cocycle  $f:G^2 \rightarrow \ZZ$ having properties that more or less encode whether or not the group can be ``arranged in a circle" in such a way that left-multiplication preserves the relative position of the elements of $G$ (See Section \ref{prelims} for a strict definition of this notion).  When restricted to the class of countable groups, circular orderability of $G$ is equivalent to the existence of an embedding of $G$ into $\mathrm{Homeo}_+(S^1)$.  

Either from the algebraic definitions or from their dynamical counterparts, it is straightforward to see that left-orderable, groups, therefore form a proper subcollection of the collection of circularly-orderable groups: every left-ordered group is circularly-orderable in a way that is akin to the way in which the one-point compactification of $\mathbb{R}$ yields a circle.  An easy way to see that the containment is proper is to observe that circularly-orderable groups can admit torsion, while left-orderable groups are torsion free.

A natural question, then, is to ask under what conditions a torsion-free circularly-orderable group is necessarily left-orderable.  This question was the focus of \cite{BCGpreprint}, which reviews the classical cohomological ways of detecting left-orderability of a given circularly-ordered group, and also offers up the following equivalence.

\begin{thm}\cite[Corollary 3.14]{BCGpreprint}
\label{thm:motivation}
A group $G$ is left-orderable if and only if $G \times \ZZ / n \ZZ $ is circularly-orderable for all $n \geq 2$.
\end{thm}

This result effectively transforms the task of left-ordering a given group $G$ into a special subcase of the more general problem of circularly-ordering direct products of groups $G \times H$ when both factors are circularly-orderable.  
While direct products are easy to deal with in certain cases (e.g. when at least one factor is left-orderable, the direct product can be circularly-ordered in a lexicographic manner), in general their behavior with respect to circular orderability is subtle.   The direct product of two circularly-orderable groups is sometimes circularly-orderable, other times not.  For example the group $\ZZ/ n\ZZ \times \ZZ/ n\ZZ$ is never circularly-orderable, since all finite circularly-orderable groups are cyclic (Proposition \ref{prop:finite-CO-is-cyclic}).  To the best of our knowledge, few results exist that address the following problem.

\begin{problem}
\label{prob:direct_products}
Given circularly-orderable groups $G$ and $H$, determine necessary and sufficient conditions that guarantee the product $G \times H$ is circularly-orderable.
\end{problem}

\subsection{Main results}

Motivated by Theorem \ref{thm:motivation}, in this paper we deal with the special case of Problem \ref{prob:direct_products} where one factor is cyclic. Call $\alpha \in H^2(G; \mathbb{Z})$ $n$-divisible if there exists $\mu \in H^2(G; \mathbb{Z})$ with $\alpha =n\mu$, and call $S \subset H^2(G; \mathbb{Z})$ fully divisible if it contains an $n$-divisible element for each $n \geq 2$. Note that $0 \in H^2(G;\ZZ)$ is fully divisible.

\begin{thm} [\text{cf.} Theorem \ref{cor:main}]
\label{thm:main_tool}
The group $G \times \ZZ/n\ZZ$ is circularly-orderable if and only if there exists a circular ordering $f$ on $G$ such that $[f] \in H^2(G;\ZZ)$ is $n$-divisible.
\end{thm}

This opens several avenues of investigation that we pursue in the later sections of this paper.  First, following \cite{BCGpreprint}, we define the \textit{obstruction spectrum} of a group $G$ to be 
\[
\Ob(G) = \{  n \in \NN_{\geq 2} \mid G \times \mathbb{Z} /n \mathbb{Z} \mbox{ is not circularly-orderable} \},
\]
which by Theorem \ref{thm:motivation} is empty if and only if $G$ is left-orderable.  We use Theorem \ref{thm:main_tool} as a tool for calculating $\Ob(G)$, from which we give a cohomological characterization of the obstruction to left-orderability of a circularly-orderable group.  

\begin{thm}[\textit{c.f.} Theorem \ref{thm:ob_characterization}]
\label{thm:LO_characterization}
Let $G$ be a group, and set 
\[ C(G) = \{ \alpha \in H^2(G; \mathbb{Z}) \mid \mbox{there exists a circular ordering $f$ of $G$ with $[f] = \alpha$} \}.
\]
Then $G$ is left-orderable if and only if $C(G)$ is fully divisible.
\end{thm}

\subsection{Applications}

As a consequence of our main results, we are able to show that for groups whose actions on $S^1$ are cohomologically rigid, the obstruction spectrum contains many elements besides the expected elements that arise from torsion within the group. In particular, let $\Mod(S_{g,1})$ be the mapping class group of a genus $g$ surface with one puncture. The group $\Mod(S_{g,1})$ is circularly-orderable, and we show the following:

\begin{thm}[\textit{c.f.} Corollary \ref{cor:mod-ob}]
Let $g \geq 2$. The direct product $\Mod(S_{g,1}) \times \ZZ/n\ZZ$ is not circularly-orderable for all $n \geq 2$.
\end{thm}

Another consequence has potential applications to the $L$-space conjecture \cite{BGW, Juhasz}. The conjecture states that for a closed, connected, orientable prime 3-manifold $M$, the following are equivalent:
\begin{itemize}
\item $M$ admits a co-oriented, taut foliation.
\item $M$ is not a Heegaard-Floer $L$-space.
\item $\pi_1(M)$ is left-orderable.
\end{itemize}
At the time of writing, the conjecture remains open for hyperbolic rational homology 3-spheres, among other classes of 3-manifolds. With this backdrop, we prove the following theorem, which is Proposition \ref{prop:L-space}.

\begin{thm}
Suppose $M$ is a connected, orientable, irreducible, rational homology 3-sphere. Suppose that $n$ is the exponent of $H_1(M;\ZZ)$. Then $\pi_1(M)$ is left-orderable if and only if $\pi_1(M) \times \ZZ/n\ZZ$ is circularly-orderable.
\end{thm}

\subsection{Direct products of circularly-orderable groups}
We finish the paper by addressing direct products of more general circularly-orderable groups.   With mild restrictions on their obstruction spectra, we are able to give a solution to Problem \ref{prob:direct_products} when one of the factors admits a circular ordering that is invariant under left and right multiplication, see Theorem \ref{thm:biCO-product}.

While circular orderability of direct products in general remains mysterious, circular orderability of iterated direct products of an amenable group turns out to behave much like circular orderability of iterated direct products of an abelian group.  That is, when an abelian group $A$ is circularly-orderable but not left-orderable, we know that $A \times A$ is not circularly-orderable since it must contain a subgroup isomorphic to $\ZZ/ n\ZZ \times \ZZ/ n\ZZ$ for some $n$.  For an amenable group $G$, it turns out that $G \times G$ may very well be circularly-orderable, but by taking successive direct products of $G$ with itself one will eventually arrive at a product $G^m = G \times \dots \times G$ that is not circularly-orderable (see Proposition \ref{eventuallynonco}, which also provides a bound for such $m$).  In fact, the Promislow group provides an example of an amenable group $G$ for which $G \times G$ is circularly-orderable, while we know from Proposition \ref{eventuallynonco} that there exists $m \geq 3$ such that $G^m$ is not.

\subsection{Outline of the paper}
Section \ref{prelims} contains the relevant background and preliminary lemmas regarding central extensions, left-orderable groups, and circularly-orderable groups. Theorem \ref{thm:main_tool} is proved in Section \ref{sec:main}, and the applications to mapping class groups and 3-manifold groups are explored in Section \ref{sec:rigidity}. The paper concludes in Section \ref{sec:direct-products} by addressing direct products of circularly-orderable groups.

\subsection{Acknowledgements}
We would like to thank Ying Hu and Justin Lanier for helpful conversations.

\section{Preliminaries}
\label{prelims}
We first review the necessary background on circular orderings and group cohomology in order to establish notation.  All group cohomology in this paper is considered with the trivial action.

\subsection{2-cocycles and central extensions} \label{sec:central-extensions}
 Let $G$ be a group and $A$ an abelian group. We begin by reviewing the constructions giving the well-known one-to-one correspondence between elements of $H^2(G;A)$ and equivalence classes of central extensions of $G$ by $A$ (see, for example, \cite[Chapter IV.3]{Brown}).

An {\it inhomogeneous 2-cocycle} $f:G^2 \to A$ is a function such that $f(id,g) = f(g,id) = 0$ for all $g \in G$ and $f(h,k) - f(gh,k) + f(g,hk) - f(g,h) = 0$ for all $g,h,k \in G$. From an inhomogeneous 2-cocycle we can construct the {\it associated central extension} $\wt G$, by equipping the set $A \times G$ with the group operation $(a,g)(b,h) = (a + b + f(g,h),gh)$. Then $\wt G$ fits into the short exact sequence
\[
1 \lra A \overset{\iota}\lra \wt G \overset{\rho}{\lra} G \lra 1
\]
where $\iota(a) = (a,id)$ and $\rho(a,g) = g$. Furthermore, $\iota(A)$ is central in $\wt G$, so $\wt G$ is a central extension of $G$ by $A$.

Conversely, suppose we have a central extension
\[
1 \lra A \overset{\iota}\lra H \overset{\rho}\lra G \lra 1.
\]
Let $\nu:G \to H$ be a {\it normalized section}, that is, a function such that $\rho \nu(g) = g$ for all $g \in G$ and $\nu(id) = id$. We can define the {\it associated cocycle} $f_\nu:G^2 \to A$ by $\iota f_\nu(g,h) = \nu(g)\nu(h)\nu(gh)^{-1}$, which is an inhomogeneous 2-cocycle.

Two central extensions
\[
1 \lra A \overset{\iota_1}\lra H_1 \overset{\rho_1}\lra G \lra 1 \eand 1 \lra A \overset{\iota_2}\lra H_2 \overset{\rho_2}\lra G \lra 1
\]
are {\it equivalent} if there exists a homomorphism $\phi:H_1 \to H_2$ such that $\phi\iota_1 = \iota_2$ and $\rho_2\phi = \rho_1$. Such a homomorphism is necessarily an isomorphism. Denote the set of equivalence classes of central extensions of $G$ by $A$ by $\vE(G,A)$. 

Given a representative extension of an element of $\vE(G,A)$, we can choose a normalized section and build the corresponding cocycle, which gives an element of $H^2(G;A)$. Conversely, given an element of $H^2(G;A)$, we can choose a representative inhomogeneous cocycle and build the corresponding central extension of $G$ by $A$, thus defining an element of $\vE(G,A)$. These constructions give a one-to-one correspondence between
$H^2(G;A)$ and  $\vE(G,A)$, taking $0$ in $H^2(G;Z)$ to the equivalence class of the extension 
\[ 1 \lra A \lra G \times A \lra G \lra 1
\]
in $\vE(G,A)$.

\subsection{Left and circular orderings}\label{sec:lo-co-basics}

 A group is {\it left-orderable} if it admits a strict total ordering $<$ invariant under left multiplication, as defined in the introduction.  When $G$ comes equipped with a left ordering $<$, the pair $(G,<)$ will be called a {\it left-ordered group}.  Equivalent to the existence of a left ordering on a group is the existence of a subset $P \subset G$ such that $G = P \sqcup P^{-1}\sqcup \{id\}$ and $P \cdot P \subset P$ called a {\it positive cone}. Given a left ordering $<$, the positive cone is given by $P = \{g \in G \mid g > id\}$. Conversely, given a positive cone $P$, the corresponding left ordering is given by $g < h$ if and only if $g^{-1}h \in P$. Elements in $P$ will be called {\it positive} elements.

We now give two definitions of a circular ordering on a group.   They turn out to be equivalent, and this equivalence is the content of Proposition \ref{prop:f-c-equivalence} below. The standard definition in the literature is Definition \ref{def:CO-hom}, but unless otherwise stated, in this paper we will use Definition \ref{def:CO-inhom} for ease of exposition.

\begin{defn}[Circular orderings as homogeneous cocycles] \label{def:CO-hom}
A {\it circular ordering} on a group $G$ is a function $c:G^3 \rightarrow \{ 0, \pm 1\}$ satisfying:
\begin{enumerate}
\item $c^{-1}(0) = \{ (g_1, g_2, g_3) \mid g_i = g_j \mbox{ for some $i \neq j$ } \}$.
\item $c$ satisfies a cocycle condition, meaning
\[c(g_2, g_3, g_4) - c(g_1, g_3, g_4) + c(g_1, g_2, g_4) - c(g_1, g_2, g_3) = 0
\]
for all $g_1, g_2, g_3, g_4 \in G$.
\item $c(g_1, g_2, g_3) = c(hg_1, hg_2, hg_3)$ for all $h, g_1, g_2, g_3 \in G$.
\end{enumerate}
\end{defn}
\begin{defn}[Circular orderings as inhomogeneous cocycles]\label{def:CO-inhom}
A {\it circular ordering} on a group $G$ is an inhomogeneous 2-cocycle $f:G^2 \to \ZZ$ satisfying:
\begin{enumerate}
\item $f(g,h) \in \{0,1\}$ for all $g,h \in G$.
\item $f(g,g^{-1}) = 1$ for all $g \in G \sm\{id\}$.
\end{enumerate}
\end{defn}

Given a group $G$, let $\vC\vH(G)$ be the set of circular orderings as in Definition \ref{def:CO-hom}, and let $\vC\vI(G)$ be the set of circular orderings as in Definition \ref{def:CO-inhom}.

Let $c \in \vC\vH(G)$. Define $f^{(c)}:G^2 \to \ZZ$ by 
\[
f^{(c)}(g,h) = \begin{cases}
0 &\text{if } g=id \text{ or }h = id,\\
1 &\text{if } gh = id \text{ and } g \neq id, \\
\frac 12(1 - c(id,g,gh)) &\text{otherwise.}
\end{cases}
\]

Let $f \in \vC\vI(G)$. Define $c^{(f)}:G^3 \to \ZZ$ by
\[
c^{(f)}(g_1,g_2,g_3) = \begin{cases}
0 &\text{if } g_i = g_j \text{ for some } i \neq j,\\
1 - 2f(g_1^{-1}g_2,g_2^{-1}g_3) &\text{otherwise.}
\end{cases}
\]

Given a circular ordering $c$ of a group $G$, the definition of $f^{(c)}$ was first introduced by \v{Z}eleva in \cite[Proof of Theorem 3]{Zeleva}.

\begin{prop}\label{prop:f-c-equivalence}
The maps $c \mapsto f^{(c)}$ and $f \mapsto c^{(f)}$ give a one-to-one correspondence $\vC\vH(G) \leftrightarrow \vC\vI(G)$.
\end{prop}
\begin{proof}
Let $c \in \vC\vH(G)$. It is clear that $f^{(c)}(g,id) = f^{(c)}(id,g) = 0$ and $f^{(c)}(g,h) \in \{0,1\}$ for all $g,h \in G$. By the definition of $f^{(c)}$, $f^{(c)}(g,g^{-1}) = 1$ for all $g \in G \sm\{id\}$. Since $f^{(c)}$ is an inhomogeneous 2-cocycle \cite[Theorem 2]{Zeleva}, $c \mapsto f^{(c)}$ is a well-defined map $\vC\vH(G) \to \vC\vI(G)$.

Conversely, let $f \in \vC\vI(G)$. By the definition, $c^{(f)}$ takes values in $\{0,\pm1\}$, and $c(g_1,g_2,g_3) = 0$ if and only if $g_i = g_j$ for some $i \neq j$. Therefore it suffices to show $c^{(f)}$ satisfies conditions (2) and (3) from Definition \ref{def:CO-hom}. Let $g_1,g_2,g_3 \in G$. If $g_i = g_j$ for some $i \neq j$, then $hg_i = hg_j$ for all $h \in G$, so $c^{(f)}(hg_1,hg_2,hg_3) = c^{(f)}(g_1,g_2,g_3) = 0$. So assume $g_1,g_2,g_3$ are distinct, and let $h \in G$. Then
\[
c^{(f)}(hg_1,hg_2,hg_3) = 1 - 2f(g_1^{-1}h^{-1}hg_2,g_2^{-1}h^{-1}hg_3) = 1 - 2f(g_1g_2^{-1},g_2g_3^{-1}) = c^{(f)}(g_1,g_2,g_3).
\]
To see the cocycle condition is satisfied, let $g_1,g_2,g_3,g_4 \in G$ and let $$\vS  = c^{(f)}(g_2,g_3,g_4) - c^{(f)}(g_1,g_3,g_4) + c^{(f)}(g_1,g_2,g_4) - c^{(f)}(g_1,g_2,g_3).$$ If all the $g_i$ are distinct, then
\[
\vS = 2(-f(g_2^{-1}g_3,g_3^{-1}g_4) + f(g_1^{-1}g_3,g_3^{-1}g_4) - f(g_1^{-1}g_2,g_2^{-1}g_4) + f(g_1^{-1}g_2,g_2^{-1}g_3)) = 0.
\]
Assume the $g_i$ are not all distinct. We deal separately with the following four cases.\\
{\bf Case 1.} ({\it $g_1 = g_2$, $g_2 = g_3$, or $g_3 = g_4$}): Then it is clear that $\vS = 0$.\\
{\bf Case 2.} ({\it $g_1 = g_3$}): First note that for any $a,b \in G$, the cocycle condition for $f$ gives $f(a^{-1},b) - f(id,b) + f(a,a^{-1}b) - f(a,a^{-1}) = 0$ so $1 - f(a,a^{-1}b) - f(a^{-1},b) = 0$. Now, we have
\[
\vS = c^{(f)}(g_2,g_1,g_4) + c^{(f)}(g_1,g_2,g_4) = 2 - 2f(g_2^{-1}g_1,(g_2^{-1}g_1)^{-1}(g_2^{-1}g_4)) - 2f((g_2^{-1}g_1)^{-1},g_2^{-1}g_4) = 0.
\]
{\bf Case 3.} ({\it $g_2 = g_4$}): As above, note that for any $a,b \in G$ we have $f(b,b^{-1}) - f(ab,b^{-1}) + f(a,id) - f(a,b) = 0$ so $f(ab,b^{-1}) + f(a,b) = 1$. Then
\[
\vS = -c^{(f)}(g_1,g_3,g_2) - c^{(f)}(g_1,g_2,g_3) = -2 + 2(f((g_1^{-1}g_2)(g_2^{-1}g_3),(g_2^{-1}g_3)^{-1}) + f(g_1^{-1}g_2,g_2^{-1}g_3)) = 0.
\]
{\bf Case 4.} ({\it $g_1 = g_4$}): For $a,b \in G$ we have $f(b,a^{-1}) - f(a,a^{-1}) + f(ab^{-1},ba^{-1}) - f(ab^{-1}, b) = 0$ so $f(b,a^{-1}) - f(ab^{-1},b) = 0$. Then
\[
\vS = c^{(f)}(g_2,g_3,g_1) - c^{(f)}(g_1,g_2,g_3) = 2(f((g_1^{-1}g_3)(g_2^{-1}g_3)^{-1},g_2^{-1}g_3) - f(g_2^{-1}g_3,(g_1^{-1}g_3)^{-1})) = 0,
\]
so we can conclude $c^{(f)} \in \vC\vH(G)$.

It now suffices to show $f^{(c^{(f)})} = f$ and $c^{(f^{(c)})} = c$ for all $f \in \vC\vI(G)$ and all $c \in \vC\vH(G)$.

If $g = id$, $h = id$, or $gh = id$, we have $f^{(c^{(f)})}(g,h) = f(g,h)$. Assume $g \neq id$, $h \neq id$, and $gh \neq id$. Then $id,g$, and $gh$ are all distinct and 
\[
f^{(c^{(f)})}(g,h) = \frac 12(1 - c^{(f)}(id,g,gh)) = f(g,h).
\]
Similarly, let $g_1,g_2,g_3 \in G$. If they are not all distinct, then $c^{(f^{(c)})}(g_1,g_2,g_3) = c(g_1,g_2,g_3) = 0$, so assume they are all distinct. Then $g_1^{-1}g_2 \neq id$, $g_2^{-1}g_3 \neq id$, and $(g_1^{-1}g_2)(g_2^{-1}g_3) \neq id$. Then
\[
c^{(f^{(c)})}(g_1,g_2,g_3) = 1 - 2f^{(c)}(g_1^{-1}g_2,g_2^{-1}g_3) = c(g_1,g_2,g_3)
\]
completing the proof.
\end{proof}

\begin{rmk}
Intuitively, a circularly-ordered group is a way of placing the elements of the group on a circle such that their order around the circle respects the group operation. A circular ordering $c \in \vC\vH(G)$ captures this by declaring that $c(g_1,g_2,g_3) = 1$ whenever one encounters the elements $(g_1,g_2,g_3)$ in that order when proceeding counterclockwise around the circle.   Similarly $c(g_1,g_2,g_3) = -1$ when they are encountered in the reverse order. Property (3) in Definition \ref{def:CO-hom} says the order of any triple around the circle is preserved under left-multiplication. For a circular ordering $f \in \vC\vI(G)$, intuitively $f(g,h) = 1$ if right multiplication by $h$ pulls $g$ around the circle in a counterclockwise direction and past the identity, and $f(g,h) = 0$ if it does not.

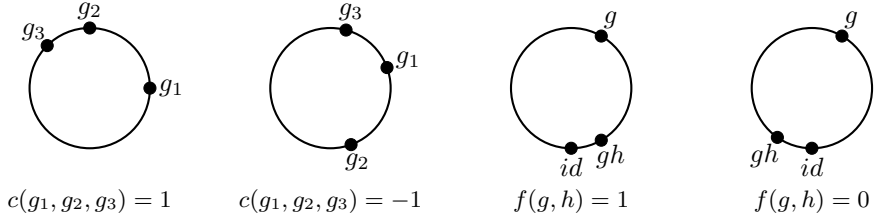
\begin{center}
\begin{tikzpicture}[scale=.8]
	\draw[thick] (-6,0) circle (1cm) node[yshift=-1.5cm]{\small $c(g_1,g_2,g_3) = 1$};
	\draw[thick] (-2,0) circle (1cm) node[yshift=-1.5cm]{\small $c(g_1,g_2,g_3) = -1$};
	\draw[thick] (2,0) circle (1cm) node[yshift=-1.5cm]{\small $f(g,h) = 1$};
	\draw[thick] (6,0) circle (1cm) node[yshift=-1.5cm]{\small $f(g,h) = 0$};
	
	\filldraw[fill=black,draw=black,xshift=-6cm] (1,0) circle (.1cm) node[right]{$g_1$};
	\filldraw[fill=black,draw=black,xshift=-6cm] (0,1) circle (.1cm) node[above]{$g_2$};
	\filldraw[fill=black,draw=black,xshift=-6cm] ({cos(135)},{sin(135)}) circle (.1cm) node[shift={({.25*cos(135)},{.25*sin(135)})}]{$g_3$};
	
	\filldraw[fill=black,draw=black,xshift=-2cm] ({cos(20)},{sin(20}) circle (.1cm) node[shift={({.3*cos(20)},{.3*sin(20)})}]{$g_1$};
	\filldraw[fill=black,draw=black,xshift=-2cm] ({cos(75)},{sin(75)}) circle (.1cm) node[shift={({.25*cos(75)},{.25*sin(75)})}]{$g_3$};
	\filldraw[fill=black,draw=black,xshift=-2cm] ({cos(290)},{sin(290)}) circle (.1cm) node[shift={({.25*cos(290)},{.25*sin(290)})}]{$g_2$};

	\filldraw[fill=black,draw=black,xshift=2cm] (0,-1) circle (.1cm) node[below]{$id$};
	\filldraw[fill=black,draw=black,xshift=2cm] ({cos(60)},{sin(60}) circle (.1cm) node[shift={({.25*cos(60)},{.25*sin(60)})}]{$g$};
	\filldraw[fill=black,draw=black,xshift=2cm] ({cos(-60)},{sin(-60)}) circle (.1cm) node[shift={({.25*cos(-60)},{.25*sin(-60)})}]{$gh$};
	
	\filldraw[fill=black,draw=black,xshift=6cm] (0,-1) circle (.1cm) node[below]{$id$};
	\filldraw[fill=black,draw=black,xshift=6cm] ({cos(60)},{sin(60)}) circle (.1cm) node[shift={({.25*cos(60)},{.25*sin(60)})}]{$g$};
	\filldraw[fill=black,draw=black,xshift=6cm] ({cos(-125)},{sin(-125)}) circle (.1cm) node[shift={({.3*cos(-125)},{.3*sin(-125)})}]{$gh$};
\end{tikzpicture}
\end{center}
Since the theory around central extensions relies on inhomogeneous 2-cocycles (see Section \ref{sec:central-extensions}), for the remainder of this paper a circular ordering on a group $G$ will be an element $f \in \vC\vI(G)$, unless explicitly stated otherwise. If a group $G$ admits a circular ordering we say $G$ is {\it circularly-orderable}. When $G$ comes equipped with a circular ordering $f$, the pair $(G,f)$ will be called a {\it circularly-ordered group}.
\end{rmk}

\subsection*{Lexicographic circular orderings}
As noted in the introduction, it is not true that circular-orderability is a property that is preserved under taking extensions. However, an extension of a circularly-ordered group by a left-ordered group can be circularly-ordered by a standard lexicographic construction, which we review here (see \cite[Lemma 2.2.12]{Calegari}).

It follows from the cocycle condition that for a homogeneous circular ordering $c$ on a group $G$, that if $\sigma \in S_3$ is a permutation, then $c(g_1,g_2,g_3) = (-1)^{\sgn(\sigma)}c(g_{\sigma(1)},g_{\sigma(2)},g_{\sigma(3)})$ for all triples $(g_1,g_2,g_3) \in G^3$. Therefore, specifying the value of $c$ on a triple specifies the value of $c$ for all permutations of that triple.

Suppose $1 \to K \to G \overset{\phi}\to H \to 1$ is a short exact sequence such that $(K,<)$ is a left-ordered group and $(H,c_H)$ is a circularly-ordered group with $c_H \in \vC\vH(H)$. Let $c_< \in \vC\vH(K)$ be defined by $c_<(k_1,k_2,k_3) = 1$, if $k_1 < k_2 < k_3$. Define the {\it lexicographic circular ordering} $c \in \vC\vH(G)$ by
\[
c(g_1,g_2,g_3) = \begin{cases}
c_H(\phi(g_1),\phi(g_2),\phi(g_3)) &\text{if $\phi(g_1),\phi(g_2)$, and $\phi(g_3)$ are all distinct},\\
c_<(g_2^{-1}g_1,id,g_1^{-1}g_2) &\text{if $\phi(g_1) = \phi(g_2) \neq \phi(g_3)$},\\
c_<(g_1^{-1}g_3,id,g_1^{-1}g_2) &\text{if $\phi(g_1) = \phi(g_2) = \phi(g_3)$}.
\end{cases}
\] 

\subsection{Left-ordered central extensions}

We review in this section two important classical constructions.

\begin{const}\cite{Zeleva}\label{const:LO-central-extension}
Let $(G,f)$ be a circularly-ordered group. Consider the central extension
\[
1 \lra \ZZ \lra \wt G_\infty \lra G \lra 1
\]
associated to $f$, and note that the subset $P \subset \wt G_\infty$ by $P = \{(a,g) \in \wt G_\infty\mid a \geq 0\}\sm\{id\}$ defines a positive cone. 
\end{const}

Let $(G,<)$ be a left-ordered group. Recall that a positive element $z \in G$ is {\it $<$-cofinal} if for every $g \in G$, there exists a $t \in \NN$ such that $z^{-t} < g < z^t$. If the left ordering is clear from context, we will simply refer to an element as cofinal.  In Construction \ref{const:LO-central-extension}, it is easily checked that $P$ is indeed a positive cone, and $(1,0) \in \wt G_\infty$ is a positive, cofinal, central element.

\begin{rmk}\label{rmk:id-implies-lo}
Suppose $(G,f)$ is a circularly-ordered group such that $[f] = 0 \in H^2(G;\ZZ)$. Then $\wt G_\infty \cong G \times \ZZ$. Therefore, the existence of a circular ordering on a group $G$ that represents the trivial cohomology class implies that $G$ is left-orderable.
\end{rmk}

Suppose now that $(G,<)$ is a left-ordered group and $z \in G$ is a positive, cofinal, central element. Then for every $g \in G$, the coset $g\langle z \rangle$ has unique coset representative $\ol g$ such that $id \leq g < z$. The representative $\ol g \in g\langle z \rangle$ is called the {\it minimal representative of $g\langle z \rangle$}.

\begin{const}\cite{Zeleva}\label{const:CO-quotient}
Let $(G,<)$ be a left-ordered group and suppose $z \in G$ is a positive, cofinal, central element. Consider the central extension
\[
1 \lra \ZZ \overset{\iota}\lra G \overset{\pi}\lra G/\langle z \rangle \lra 1
\]
where $\iota(1) = z$. Let $\eta:G/\langle z \rangle \to G$ be the normalized section given by $\eta(g\langle z \rangle) = \ol g$. Then the associated cocycle $f_\eta:(G/\langle z \rangle)^2 \to \ZZ$ is a circular ordering on $G/\langle z \rangle$.
\end{const}

\subsection{Finite cyclic subgroups}
Finite cyclic subgroups, and extensions of circularly-ordered groups by finite cyclic groups, play an important part in our main result. In this section, we collect some results concerning circular orderings and finite cyclic subgroups which may be of independent interest.

While it is a well-known fact that every finite circularly-orderable group is cyclic, we begin by giving a quick algebraic proof of this fact for completeness.

\begin{prop}\label{prop:finite-CO-is-cyclic}
Let $G$ be a finite circularly-orderable group. Then $G$ is cyclic.
\end{prop}
\begin{proof}
Fix a circular ordering on $G$ and consider the left-ordered central extension $\wt G_\infty$ from Construction \ref{const:LO-central-extension}. Then $\wt G_\infty$ is torsion-free and has a finite-index cyclic subgroup, and is therefore cyclic %\cite[Lemma 3.2]{Macpherson96} 
\cite[Theorem 3]{Stallings}. Since quotients of cyclic groups are cyclic, $G$ is cyclic.
\end{proof}

It follows from the proof of Proposition \ref{prop:finite-CO-is-cyclic} that if $(G,f)$ is a finite, cyclic, circularly-ordered group, the left-ordered central extension $\wt G_\infty$ is infinite cyclic. Therefore, there exists a unique positive generator of $\wt G_\infty$.

\begin{defn}[Minimal generator]\label{def:minimal-generator}
Let $(G,f)$ be a finite, cyclic, circularly-ordered group. Consider the left-ordered central extension
\[
1 \lra \ZZ \lra \wt G_\infty \overset{\rho}\lra G \lra 1.
\]
Let $\tilde z \in \wt G_\infty$ be the unique positive generator of $\wt G_\infty$. The {\it minimal generator of $(G,f)$} is the element $z = \rho(\tilde z) \in G$.
\end{defn}

Alternatively, the minimal generator of $(G,f)$ can be defined as the unique element $z \in G$ such that $f(z,g) = 0$ for all $g \in G \sm\{z^{-1}\}$, or equivalently, such that $c^{(f)}(id,z,g) \in \{0,1\}$ for all $g \in G$.

We finish this section with a result about normal finite cyclic subgroups of circularly-orderable groups that may be of independent interest.

\begin{prop}\label{prop:central-cyclic-subgroup}
Let $(G,f)$ be a circularly-ordered group, and let $K \triangleleft G$ be a normal finite cyclic subgroup. Consider the left-ordered central extension
\[
1 \lra \ZZ \overset{\iota}\lra \wt G_\infty \overset{\rho}\lra G \lra 1.
\]
Then $\wt K = \rho^{-1}(K)$ is central in $\wt G_\infty$ and $K$ is central in $G$.
\end{prop}
\begin{proof}
Let $k \in \wt K$ and $g \in \wt G_\infty$. Since $K$ is a normal subgroup of $G$, $\wt K$ is a normal subgroup of $\wt G_\infty$. Since $\wt K$ is an infinite cyclic group, there exists $t \in \NN$ such that $k^t \in \iota(\ZZ)$. Therefore, $gk^tg^{-1} = k^t$; and since $gkg^{-1} \in \wt K$, we must have $gkg^{-1} = k$. We conclude that $\wt K$ is central in $\wt G_\infty$ and $K$ is central in $G$.
\end{proof}

\section{Circular orderability of finite cyclic central extensions}\label{sec:main}
The goal of this section is to prove Theorem \ref{cor:main}. We first set some notation. For $n \geq 2$, let $\mathfrak{p}_n:\ZZ \to \ZZ/n\ZZ$ be the quotient map. For $a \in \ZZ$, we will denote the coset $a + n\ZZ$ by $[a] \in \ZZ/n\ZZ$. Suppose $(G,f)$ is a circularly-ordered group. Then for each $n \geq 2$, $\mathfrak p_nf:G^2 \to \ZZ/n\ZZ$ is an inhomogeneous 2-cocycle. We will denote the associated central extension by $\wt G_n$.

%{\red Maybe some words motivating the idea of what's coming, motivated by a picture and covering spaces of $S^1$}.

\begin{prop}\label{prop:circularly-orderable-finite-extension}
Let $(G,f)$ be a circularly-ordered group. Let $n \geq 2$, and consider the inhomogeneous 2-cocycle $\mathfrak p_nf :G^2 \to \ZZ/n\ZZ$. Then the associated central extension $\wt G_n$ is circularly-orderable.
\end{prop}
\begin{proof}
Let $\wt G_\infty$ be the left-ordered central extension of $G$ corresponding to $f$ from Construction \ref{const:LO-central-extension}, and let $z = (1,id) \in \wt G_\infty$ be the positive, cofinal, central element. Then $z^n = (n,id)$ is a positive, cofinal, central element, so by Construction \ref{const:CO-quotient}, $\wt G_\infty/\langle z^n \rangle$ is circularly-orderable. Define a map $\varphi:\wt G_\infty \to \wt G_n$ by $\varphi((a,g)) = (\mathfrak p_n(a),g)$. Then $\varphi$ is a surjective homomorphism, and it is easy to check that $(a,g) \in \ker(\varphi)$ if and only if $g = id$ and $a = nk$ for some $k \in \ZZ$. Therefore, $\ker(\varphi)= \langle z^n\rangle$ and $\wt G_\infty/\langle z^n \rangle \cong \wt G_n$, completing the proof.
\end{proof}

\begin{rmk}
We can explicitly write down the circular ordering on $\wt G_n$ arising from the proof of Proposition \ref{prop:circularly-orderable-finite-extension} as follows. Let $s:\ZZ/n\ZZ \to \ZZ$ be the section given by $s([a]) = a$, where $0 \leq a < n$. One can check  that $f_s:(\ZZ/n\ZZ)^2 \to \ZZ$ is a circular ordering on $\ZZ/n\ZZ$ (in fact $f_s$ is the circular ordering coming from the embedding $\ZZ/n\ZZ \hookrightarrow S^1$ where $[a] \mapsto e^{2\pi ia/n}$). Then the circular ordering $\hat f$ on $\wt G_n$ is given by
\[
\hat f ((a_1,g_1),(a_2,g_2)) = \begin{cases}
f_s(a_1,a_2) &\text{if } s(a_1) + s(a_2) \neq n-1 \\
f(g_1,g_2) &\text{otherwise.}
\end{cases}
\]
\end{rmk}

\begin{prop}\label{prop:finite-CO-quotient}
Let $(G,f)$ be a circularly-ordered group with central subgroup $K \cong \ZZ/n\ZZ$, and let $z \in K$ be the minimal generator. Consider the central extension
\[
1 \lra \ZZ/n\ZZ \overset{\iota}\lra G \overset{\pi}\lra G/K \lra 1
\]
where $\iota([1]) = z$ and $\pi$ is the quotient map. There exists a circular ordering $\bar f:(G/K)^2 \to \ZZ$ such that $\mathfrak p_n\bar f = f_\nu$ for some section $\nu:G/K \to G$ of the above central extension.
\end{prop}
\begin{proof}
Consider the left-ordered central extension
\[
1 \lra \ZZ \overset{\epsilon}\lra \wt G_\infty \overset{\rho}\lra G \lra 1
\]
corresponding to $f$, and let $<$ be the left ordering on $\wt G_\infty$. Let $\wt K = \rho^{-1}(K)$, and let $\tilde z \in \rho^{-1}(z)$ be the positive generator of $\wt K$. By Proposition \ref{prop:central-cyclic-subgroup}, $\tilde z$ is central. Furthermore, since $(1,0) \in \wt K$, $\tilde z$ is cofinal and positive in $(\wt G_\infty,<)$. Consider the central extension
 \[
1 \lra \ZZ \overset{\hat\epsilon}\lra \wt G_\infty \overset{\hat\rho}\lra \wt G_\infty/\wt K \lra 1
\]
where $\hat\epsilon(1) = \tilde z$ and $\hat\rho$ is the quotient map. Let $\eta:\wt G_\infty/\wt K \to \wt G_\infty$ be the section given by $\eta((a,g)\wt K)) = \tilde g$, where $\tilde g \in \hat\rho^{-1}((a,g)\wt K)$ is the unique element such that $id \leq \tilde g < \tilde z$. Then by Construction \ref{const:CO-quotient}, the associated cocycle $f_\eta$ is a circular ordering on $\wt G_\infty/\wt K$.

Now let $\phi:G/K \to \wt G_\infty/\wt K$ be the isomorphism given by $\phi(gK) = (0,g)\wt K$. Define $\bar f(gK,hK) = f_\eta(\phi(gK),\phi(hK))$, which is a circular ordering on $G/K$ since $\phi$ is an isomorphism.

Define $\nu:G/K \to G$ by $\nu = \rho\eta\phi$. We first wish to show $\nu$ is a normalized section of the quotient map $\pi:G \to G/K$. To that end, note that $\eta((a,g)\wt K) = (0,g')$ for some $g' \in G$. Then
\[
\pi\nu(gK) = \pi\rho\eta\phi(gK) = \pi\rho\eta((0,g)\wt K) = \pi\rho((0,g)) = \pi(g) = gK
\] 
and $\nu(K) = \rho\eta\phi(K) = \rho\eta(\wt K) = \rho((0,id)) = id$, so $\nu$ is a normalized section of $\pi$.

It remains to show that $\mathfrak p_n \bar f = f_\nu$. First note that for $t \in \ZZ$, $\iota\mathfrak p_n(t) = \iota([t]) = z^t$ and $\rho\hat\epsilon(t) = \rho(\tilde z^t) = z^t$. Therefore, $\iota \mathfrak p_n = \rho \hat \epsilon: \ZZ \to G$. For $gK,hK \in G/K$ we have
\begin{align*}
\iota\mathfrak p_n\bar f(gK,hK) &= \iota\mathfrak p_nf_\eta(\phi(gk),\phi(hK)) \\
&= \rho\hat\epsilon f_\eta(\phi(gK),\phi(hK)) \\
&= \rho\left(\eta\phi(gK)\eta\phi(hK)(\eta\phi(ghK))^{-1}\right)\\
&= \nu(gK)\nu(hK)\nu(ghK)^{-1}\\
&=\iota f_\nu(gK,hK).
\end{align*}
Since $\iota$ is injective, $\mathfrak p_n\bar f = f_\nu$, completing the proof.
\end{proof}

While the next corollary does not appear elsewhere in this paper, it is worth noting as it is a somewhat surprising and novel consequence.
%{\red Maybe the next corollary doesn't belong here, and should be mentioned in the introduction.}
\begin{cor}
Suppose $G$ is a circularly-orderable group and $K$ is a finite cyclic normal subgroup of $G$. Then $G/K$ is circularly-orderable.
\end{cor}
\begin{proof}
By Proposition \ref{prop:central-cyclic-subgroup}, $K$ is central; so, by Proposition \ref{prop:finite-CO-quotient}, $G/K$ is circularly-orderable.
\end{proof}

For a group $G$ and an integer $n \geq 2$, let $\mathfrak P_n:H^2(G;\ZZ) \to H^2(G;\ZZ/n\ZZ)$ be the homomorphism given by $\fP_n([f]) = [\mathfrak p_n f]$. 

\begin{lem}\label{thm:main}
Let $G$ be a group. The group $G \times \ZZ/n\ZZ$ is circularly-orderable if and only if there exists a circular ordering $f$ of $G$ such that $[f] \in \ker(\fP_n)$.
\end{lem}
\begin{proof}
Suppose $G \times \ZZ/n\ZZ$ is circularly-orderable. Then by Proposition \ref{prop:finite-CO-quotient}, there exists a circular ordering $f$ of $G$ such that $\mathfrak p_n f = f_\nu$ for some section $\nu$ of the short exact sequence $1 \to \ZZ/n\ZZ \to G \times \ZZ/n\ZZ \to G \to 1$. However, the short exact sequence is split, so $[f_\nu] = 0$ in $H^2(G;\ZZ/n\ZZ)$. Therefore $[f] \in \ker(\fP_n)$.

Conversely, suppose there is a circular ordering $f$ on $G$ such that $[\mathfrak p_n f] = 0 \in H^2(G;\ZZ/n\ZZ)$. By Proposition \ref{prop:circularly-orderable-finite-extension}, the central extension $\wt G_n$ associated to $\mathfrak p_n f$ is circularly-orderable. Since $[\mathfrak p_n f]$ is the trivial cohomology class, $\wt G_n \cong G \times \ZZ/n\ZZ$.
\end{proof}

We are now ready to prove the main theorem. 

\begin{thm}\label{cor:main}
The group $G \times \ZZ/n\ZZ$ is circularly-orderable if and only if there exists a circular ordering $f$ on $G$ and an element $\mu \in H^2(G;\ZZ)$ such that $n\mu = [f]$.
\end{thm}
\begin{proof}
By Lemma \ref{thm:main}, it suffices to show that $[f] \in \ker(\fP_n)$ if and only if there exists a $\mu$ such that $n\mu = [f]$. Consider the short exact sequence
\[
1 \lra \ZZ \overset{\mathfrak i_n}\lra \ZZ \overset{\mathfrak p_n}\lra \ZZ/n\ZZ \lra 1
\]
where $\mathfrak i_n(a) = na$. Let $\fI_n:H^2(G;\ZZ) \to H^2(G;\ZZ)$ be the homomorphism given by $\fI_n([f]) = [\mathfrak i_n f]$. By the long exact sequence of cohomology, $\im(\fI_n) = \ker (\fP_n)$. The proof is completed by observing that $\fI_n\mu = n\mu$ for all $\mu \in H^2(G;\ZZ)$. 
\end{proof}

\section{Left-orderability, rigidity and the obstruction spectrum}\label{sec:rigidity}

Let $G$ be a circularly-orderable group. Recall from the introduction that the {\it obstruction spectrum of $G$} is given by
\[
\Ob(G) = \{n \in \NN_{\geq 2} \mid G \times \ZZ/n\ZZ \text{ is not circularly-orderable}\}.
\]
By Theorem \ref{thm:motivation}, $\Ob(G)$ is empty if and only if $G$ is left-orderable.

We begin by cohomologically characterizing the obstruction spectrum and giving some tools to compute it in several cases.  First, by Theorem \ref{cor:main} we arrive at the following, which also proves Theorem \ref{thm:LO_characterization} from the introduction. Recall that an element $\alpha \in H^2(G;\ZZ)$ is {\it $n$-divisible} if there exists $\mu \in H^2(G;\ZZ)$ such that $n\mu = \alpha$.

\begin{thm}
\label{thm:ob_characterization}
Let $G$ be a group. If 
\[ C(G) = \{ \alpha \in H^2(G; \mathbb{Z}) \mid \mbox{there exists a circular ordering of $G$ with $[f] = \alpha$} \},
\]
then $$\mathrm{Ob}(G) = \{ n \in  \NN_{\geq 2} \mid \mbox{there is no $n$-divisible element of $C(G)$} \}.$$ 
\end{thm}

In light of this, certain elements will always appear in the obstruction spectrum of  given group.  Recall that the {\it exponent} $e$ of a group $A$ is the smallest positive integer such that $a^e = id$ for all $a \in A$. In our case, $A$ will be an abelian group, so we will write $a^e$ as $ea$. 

\begin{prop}
\label{prop:obs_facts}
Let $G$ be a circularly-orderable group such that $H^2(G;\ZZ)$ has finite exponent $e \geq 2$.
\begin{enumerate}
\item If $\gcd(n,e) = 1$, then $n \notin \Ob(G)$.
\item If $G$ is not left-orderable, then $e \in \Ob(G)$.
\item If $G$ is not left-orderable and $e$ is prime, then $\Ob(G) = e\NN$.
\end{enumerate} 
\end{prop}
\begin{proof}
For (1), let $f$ be a circular ordering of $G$. Let $x,y \in \ZZ$ be such that $nx + ey = 1$. Then $[f] = (nx + ey)[f] = n(x[f])$, so by Theorem \ref{cor:main}, $n \notin \Ob(G)$.

For (2), note that if $\sigma = e\mu$ in $H^2(G;\ZZ)$, then $\sigma = 0$. However, since $G$ is not left-orderable, there is no circular ordering $f$ such that $[f] = 0 \in H^2(G;\ZZ)$ (see Remark \ref{rmk:id-implies-lo}). Therefore, by Theorem \ref{cor:main}, $e \in \Ob(G)$.

Statement (3) follows immediately from (1) and (2).
\end{proof}

\subsection{Rigidity and mapping class groups.}  
\label{mappingclass}
Let $S_{g,1}$ denote a genus $g\geq 2$ surface with one marked point $p$, and let $\Mod(S_{g,1})$ denote the mapping class group.  Recall that $\Mod(S_{g,1})$ acts on $S^1$ as follows.

Fix a hyperbolic metric on $S_{g,1}$ and choose a lift $\tilde{p} \in \mathbb{H}^2$ of $p$.  Each $\phi \in \Mod(S_{g,1})$ can be represented by $f: S_{g, 1} \rightarrow S_{g,1}$ with $f(p) =p$, and each such $f$ admits a unique lift $\tilde{f}: \mathbb{H}^2 \rightarrow \mathbb{H}^2$ with $\tilde{f}( \tilde{p}) = \tilde{p}$.  This gives an action of $\Mod(S_{g,1})$ on $\mathbb{H}^2$, which extends to an action on  $\partial \mathbb{H}^2 \cong S^1$.  Call the resulting action of $\Mod(S_{g,1})$ on $S^1$ the \textit{standard action}.\footnote{Changing the choice of hyperbolic metric or of $p \in S_{g,1}$ changes the resulting action on $S^1$ to a new action that semiconjugate to the original.  Therefore the standard action is, according to our definition, only defined up to semiconjugacy.}

Here is the usual recipe for creating a circular ordering $f$ of a subgroup $G$ of $\mathrm{Homeo}_+(S^1)$.  Fix an enumeration of a countable dense subset of $S^1$, say $\{r_0, r_1, \dots \}$ and let $f_{circ}$ denote the usual circular ordering of $S^1$.  Given $g_1, g_2 \in G$, let $m$ denote the smallest index such that $g_1(r_m) \neq g_2(r_m)$. Set $f(g_1, g_2) = f_{circ}(g_1(r_m), g_2(r_m))$.  Evidently, different choices of countable dense subset or different choices of enumeration will yield different circular orderings, however the resulting circular orderings will all represent the same class in bounded cohomology. %{\red There are quite a few unproven claims here about how one can jump back and forth between circular orderings and cocycles, a lot is swept under the rug.  Is this OK?}

Let $f_{s}$ denote a circular ordering of $\Mod(S_{g,1})$ arising from the standard action via the recipe described in the previous paragraph.  Recently Mann and Wolff proved the following rigidity result.

\begin{thm}\cite{MWpreprint}
Let $g \geq 2$. If $f$ is any circular ordering of $\Mod(S_{g,1})$, then $[f] = [f_s]$ in $H^2_b(\Mod(S_{g,1}); \ZZ)$.
\end{thm} 

Recall that for every group $G$ there is a map $H^2_b(G; \ZZ) \lra H^2(G; \ZZ)$.  When $G = \Mod(S_{g,1})$, the image of $[f_s]$ under this map is computed in \cite[Sections 5.5.3--5.5.5]{FM12} as the Euler class of the central extension 
\[ 1 \lra \ZZ \lra \widetilde{\Mod}(S_{g,1}) \lra \Mod(S_{g,1}) \lra 1.
\]
They show that it is a primitive element of $H^2(\Mod(S_{g,1}); \ZZ) \cong \ZZ^2$.  Combining this result with the work of Mann and Wolff, we have the following.

\begin{cor}
Let $g \geq 2$. If $f$ is any circular ordering of $\Mod(S_{g,1})$, then $[f] = [f_s]$ is a primitive element of $H^2(\Mod(S_{g,1}); \ZZ)$.
\end{cor}

Combining this fact with Theorem \ref{thm:ob_characterization}, we immediately arrive at the following.

\begin{cor}\label{cor:mod-ob}
Let $g \geq 2$. With notation as above,
$\Ob(\Mod(S_{g,1})) = \NN_{\geq 2}$.
\end{cor}

Note that $\Mod(S_{g,1})$ contains torsion elements, and thus $m \NN \subset \Ob(\Mod(S_{g,1}))$ for every $m \in  \NN_{\geq 2}$ that is not relatively prime to the order of every torsion element in $\Mod(S_{g,1})$.  However, the order of the torsion elements in $\Mod(S_{g,1})$ is bounded above by $4g+2$.  Thus $\Ob(\Mod(S_{g,1})) = \NN_{\geq 2}$ is in fact a reflection of the rigidity results of Mann and Wolff, and not a consequence of torsion.

\subsection{Fundamental groups of $3$-manifolds.}

We begin by noting a potential application to the $L$-space conjecture \cite{BGW, Juhasz}.

\begin{prop}\label{prop:L-space}
Suppose $M$ is a connected, orientable, irreducible, rational homology 3-sphere. Suppose that $n$ is the exponent of $H_1(M;\ZZ)$. Then $\pi_1(M)$ is left-orderable if and only if $\pi_1(M) \times \ZZ/n\ZZ$ is circularly-orderable.
\end{prop}
\begin{proof}
If $\pi_1(M)$ is left-orderable, then $\pi_1(M) \times \ZZ/n\ZZ$ is circularly-orderable by a lexicographic argument.
Conversely, if $\pi_1(M)$ is finite and $\pi_1(M) \times \ZZ/n\ZZ$ is circularly-orderable, it must be that $\pi_1(M)$ is cyclic. However, it must be cyclic of order $n$, a contradiction. Therefore we may assume $\pi_1(M)$ is infinite. In this case, $M$ is a $K(\pi_1(M),1)$ and $H^2(\pi_1(M);\ZZ) \cong H^2(M;\ZZ) \cong H_1(M;\ZZ)$ by Poincar\'e duality. Thus, the exponent of $H^2(\pi_1(M);\ZZ)$ is $n$. Since $\pi_1(M) \times \ZZ/n\ZZ$ is circularly-orderable, $n \notin \Ob(\pi_1(M))$. Therefore by Proposition \ref{prop:obs_facts}(2), $\pi_1(M)$ is left-orderable.
\end{proof}

Using Proposition \ref{prop:obs_facts}, we can now deal with many of the 3-manifold examples raised in \cite{BCGpreprint}.

\begin{prop}
\label{3manifolds}
Suppose that $M$ is an irreducible rational homology $3$-sphere such that $\pi_1(M)$ is not left-orderable, but is circularly-orderable.  If the exponent of $H_1(M ; \ZZ)$ is a prime $p$, then $\Ob(\pi_1(M)) = p \NN$.
\end{prop}
\begin{proof}
If $\pi_1(M)$ is infinite and $M$ is irreducible, then $H^2(\pi_1(M); \ZZ) \cong H^2(M; \ZZ) \cong H_1(M; \ZZ)$.  The result then follows from Proposition \ref{prop:obs_facts}(3).  On the other hand if $\pi_1(M)$ is finite then it is cyclic of order $p$, and the result follows.
\end{proof}

When $H_1(M ; \ZZ)$ is cyclic of prime order $p$, it is possible to give a straightforward description of the circular orderings of the products $\pi_1(M) \times \ZZ / n\ZZ$ when $n$ is not a multiple of $p$.  We describe them in the lemma and proposition below, but in the more general setting of a group $G$ having certain restrictions on its homology.

\begin{lem}
Let $G$ be a group with $H_2(G; \ZZ) = 0$.  If $G$ is circularly-orderable, then $[G, G]$ is left-orderable.
\label{lemma:locommutator}
\end{lem}
\begin{proof}
There is a left-ordered central extension
\[
1 \lra \ZZ \lra \wt G_\infty \lra G \lra 1,
\]
and by \cite[Lemma 3.10]{BRW05}, the identity $I_G: G \rightarrow G$ yields an injective lift $\widetilde{I_G}|_{[G, G]}:[G, G] \lra \wt G_\infty$.  
\end{proof}

\begin{prop}
Suppose that $G$ is a circularly-orderable group that is not left-orderable.  If $H_2(G; \ZZ) = 0$ and $H_1(G; \ZZ)$ is cyclic of order $n$, then 
\[ \Ob(G) \subseteq \{ d \in \mathbb{N}_{\geq 2} \mid \gcd(n, d) \neq 1 \}.
\]
\end{prop}
\begin{proof}
First, observe that since $G$ satisfies the hypotheses of Lemma \ref{lemma:locommutator}, we know that $[G, G]$ is left-orderable.  Therefore if $H_1(G; \ZZ)$ is cyclic of order $n$ and $\gcd(n, d) = 1$ we can construct a circular ordering of $G \times \ZZ / d \ZZ$ lexicographically from the short exact sequence
\[ 1 \lra [G, G] \lra G \times \ZZ / d\ZZ  \lra H_1(G; \ZZ) \times \ZZ / d\ZZ \lra 1,
\]
since $H_1(G; \ZZ) \times \ZZ / d\ZZ$ is cyclic.  
\end{proof}

In particular when $M$ (as in Proposition \ref{3manifolds}) has cyclic first homology of order $d$, it is via short exact sequences that we can construct the circular orderings of $\pi_1(M) \times \ZZ/n\ZZ$ for each $n \notin \Ob(\pi_1(M))$ with $\gcd(d,n)=1$ that witness $n$ as being outside of the obstruction spectrum.

\section{Direct products}\label{sec:direct-products}

\subsection{The direct product of two circularly-orderable groups}

In this section, we show how to use the conditions developed in the previous sections to circularly-order various direct products of two groups. Note that $\Ob(G)$ has the structure of a poset, where $n \leq m$ if and only if $n$ divides $m$.  For such a poset, we denote the set of minimal elements by $\min( \Ob(G))$.

We begin with the case of one factor being a circularly-orderable group which admits a bi-invariant circular ordering. A {\it bi-invariant circular ordering} on a group $G$ is a circular ordering $c:G^3 \to \{0,\pm1\}$ such that $c(g_1,g_2,g_3) = c(hg_1,hg_2,hg_3) = c(g_1h,g_2h,g_3h)$ for all $g_1,g_2,g_3,h \in G$. Similarly, a {\it bi-orderable group} is a group admitting a strict total ordering that is invariant under group multiplication from both the right and left.

Recall from \cite{BCGpreprint}, the {\it torsion part} of the obstruction spectrum of a circularly-orderable group $G$ is defined to be
\[
\Ob_T(G) = \{n \in \NN_{\geq 2} \mid \text{there exists } g\in G \text{ such that } g^k = id \text{ and } \gcd(k,n) \neq 1\}.
\]
Note that $\Ob_T(G)$ is a subset of $\Ob(G)$, and it is all the elements of the obstruction spectrum that detect the presence of torsion in $G$.

When $G$ admits a bi-invariant circular ordering, we have $\Ob_T(G) = \Ob(G)$ \cite[Proposition 4.1]{BCGpreprint}. In particular, $\min(\Ob(G))$ consists entirely of primes, and for every prime $p \in \min(\Ob(G))$, there exists an element $g \in G$ of order $p$.

\begin{thm}\label{thm:biCO-product}
Let $A$ be a group admitting a bi-invariant circular ordering, and let $G$ be a circularly-orderable group with the property that $\min(\Ob(G))$ consists only of prime numbers. Then $G \times A$ is circularly-orderable if and only if $\min(\Ob(G)) \cap \min(\Ob(A)) = \emptyset$.
\end{thm}
\begin{proof}
Suppose $p \in \min(\Ob(G)) \cap \min(\Ob(A))$. Then since $A$ admits a bi-invariant circular ordering, $A$ contains an element of order $p$. Therefore $G \times A$ contains a subgroup isomorphic to $G \times \ZZ/p\ZZ$, which is not circularly-orderable, so we conclude $G \times A$ is not circularly-orderable.

Conversely, suppose $\min(\Ob(G)) \cap \min(\Ob(A)) = \emptyset$, and let $H$ be an arbitrary finitely-generated subgroup of $G \times A$. It suffices to show that $H$ is circularly-orderable, since if every finitely-generated subgroup of a group is circularly-orderable, the group is circularly-orderable \cite[Lemma 2.14]{Clay}.

Since $A$ admits a bi-invariant circular ordering, $A$ is a subgroup of $\Gamma \times S^1$, where $\Gamma$ is a bi-orderable group \cite{Sw59}. Therefore, every finitely-generated subgroup of $A$ is contained in a subgroup of the form $F \times \ZZ/m\ZZ$ for some $m \in \NN$, where $F$ is a bi-orderable group. We now have that $H$ is contained in a subgroup of $G \times A$ of the form $G \times F \times \ZZ/m\ZZ$, where $F$ is a bi-orderable group and $m$ is the order of some element of $A$. If $G \times \ZZ/m\ZZ$ is circularly-orderable, then we may conclude by a lexicographic construction that $H$ is circularly-orderable. Thus, it suffices to show that $G \times \ZZ/m\ZZ$ is circularly-orderable. 

If $m = 1$ we are done, so assume $m \geq 2$ and let $p$ be a prime dividing $m$. Note that for all such $p$, since $p \in \min(\Ob(A))$ the group $G \times \ZZ/p\ZZ$ is circularly-orderable. Therefore, if $G \times \ZZ/m\ZZ$ were not circularly-orderable, there would be an element of $\min(\Ob(G))$ that divides $m$ but is not prime, contradicting the assumption on $G$. We may finally conclude $G \times \ZZ/m\ZZ$ is circularly-orderable.
\end{proof}

From Corollary \ref{cor:mod-ob}, for the mapping class group $\Mod(S_{g,1})$, we have that $\min(\Ob(\Mod(S_{g,1})))$ consists only of prime numbers (in fact it is equal to the set of all prime numbers). The next corollary follows immediately.

\begin{cor}\label{cor:MCG-biCO-product}
Let $G$ be a group admitting a bi-invariant circular ordering, but not a left ordering. Then for $g \geq 2$, $\Mod(S_{g,1}) \times G$ is not circularly-orderable. 
\end{cor}

Since $\min(\Ob(G))$ consists only of primes for any group $G$ admitting a bi-invariant circular ordering, we obtain the next corollary.

\begin{cor}\label{cor:biCO-product}
Let $G$ and $H$ be groups admitting a bi-invariant circular ordering. Then $G \times H$ is circularly-orderable if and only if $\min(\Ob(G)) \cap \min(\Ob(H)) = \emptyset$.
\end{cor}

Corollary \ref{cor:biCO-product} suggests that there may be a way to characterize, via obstruction spectra, when the direct product of two circularly-orderable groups is circularly-orderable. The next example shows this is not possible.

\begin{eg}\label{eg:Prom1}
Let $G$ be the Promislow group, also known as the fundamental group of the Hantzsche-Wendt manifold. It was shown in \cite[Example 4.14]{BCGpreprint} that $G$ is circularly-orderable, and $\Ob(G) = 4\NN$. In particular, $G \times \ZZ/2\ZZ$ is circularly-orderable, whereas $G \times \ZZ/4\ZZ$ is not. However, $\Ob(\ZZ/2\ZZ) = \Ob(\ZZ/4\ZZ)$, so there is no way to characterize when the direct product of two circularly-orderable groups is circularly-orderable in terms of the obstruction spectra of the two groups. This example also shows that the assumption on the set of minimal elements in Theorem \ref{thm:biCO-product} cannot be dropped.
\end{eg}

A characterization of when the direct product of two circularly-orderable groups is circularly-orderable remains tantilizingly out of reach.

\subsection{Iterated direct products}

We finish the paper by investigating circular orderability of iterated direct products of a circularly-orderable group with itself.   Let us first record the following special case of Corollary \ref{cor:biCO-product}, which follows from the fact that a circularly-orderable group $G$ is left orderable if and only if $\Ob(G) = \emptyset$.

\begin{prop}\label{prop:iterated-biCO}
Let $G$ be a group admitting a bi-invariant circular ordering. Then $G \times G$ is circularly-orderable if and only if $G$ is left-orderable.
\end{prop}

In a lecture given by the second author, Dani Wise asked if
 Proposition \ref{prop:iterated-biCO} might hold for general circularly-orderable groups. We now provide a counterexample.

\begin{eg}\label{eg:Prom2}
Let $G$ be the Promislow group, which we used in Example \ref{eg:Prom1}. This group admits the presentation $G = \langle a, b \mid ab^2a^{-1}b^2, ba^2b^{-1}a^2 \rangle$. There is a homomorphism $\phi : G \rightarrow \ZZ / 2\ZZ$ given by $\phi(a) = 1, \phi(b) = 0$, whose kernel is isomorphic to a semidirect product $(\ZZ \times \ZZ) \ltimes \ZZ$ and is thus left-orderable.  Therefore, $G$ is circularly-orderable by a lexicographic argument.  The group $G \times \ZZ /2 \ZZ$ is also circularly-orderable, since the obstruction spectrum of $G$ is $4 \NN$ by \cite[Example 4.14]{BCGpreprint}.

Thus, the short exact sequence 
\[ 1 \longrightarrow K \longrightarrow G \times G \stackrel{\id \times \phi}{\longrightarrow} G \times \ZZ /2 \ZZ \longrightarrow 1
\]
yields a lexicographic circular-ordering of $G \times G$. Finally, note that since $\Ob(G) \neq \emptyset$, $G$ is not left-orderable.
\end{eg}

Intriguingly, since the Promislow group $G$ is amenable (it contains an abelian subgroup of finite index), the next proposition shows that there exists a $n \geq 3$ such that $G^n$ is not circularly-orderable.

\begin{prop}
\label{eventuallynonco}
Assume $G$ is a finitely generated circularly-orderable amenable group which is not left-orderable.  If $G/G'$ is finite, then there exists $n \geq 2$ such that the $n$-fold direct product $G\times \dots \times G$ is not circularly-orderable.
\end{prop}
\begin{proof}
Since $G/G'$ is finite, $G$ has only finitely many cyclic quotients.  Suppose they are $\{C_i\}_{i=1}^m$, and let $e$ denote the exponent of the abelian group $\bigoplus_{i=1}^m C_i$. 

Suppose the $me$-fold direct product $G^{me}$ is circularly-orderable, and fix a chosen circular ordering of it.  Then as $G$ is amenable, so is $G^{me}$, which means that the chosen circular ordering of $G^{me}$ arises lexicographically from a short exact sequence
\[ 1 \longrightarrow K \longrightarrow G^{me} \stackrel{\mathrm{rot}}{\longrightarrow} C \rightarrow 1,
\]
where $K$ is left-orderable, $C$ is cyclic, and the map $\mathrm{rot}:G^{me} \rightarrow C$ is the rotation number homomorphism induced by the circular ordering of $G^{me}$ \cite[Proposition 4.8]{BCGpreprint}.   Let $\phi_i: G \rightarrow G^{me}$ denote the inclusion map $g \mapsto (id, \dots, g, \dots, id)$, where the non-identity term appears in the $i$-th position.

Choose indices $j_1, j_2, \ldots, j_e$ such that $\mathrm{rot} \circ \phi_{j_a} =\mathrm{rot} \circ \phi_{j_b}$ for all $a, b \in \{1, \ldots, e\}$. This is possible because there are only $m$ cyclic quotients of $G$.   For ease of exposition, suppose that $j_1 = 1, j_2 = 2, \ldots, j_e = e$.  Then consider the inclusion $\iota: G \rightarrow G^{me}$ defined by $g \mapsto (g,  \ldots, g, id, \ldots, id)$, where the first $e$ terms are equal to $g$ and the rest are the identity.  Now we compute that for every $g \in G$
\[ \mathrm{rot} \circ \iota (g) = \sum_{i=1}^e \mathrm{rot} \circ \phi_i (g) = 0,
\]
since $\mathrm{rot} \circ \phi_{i}(g)$ are all equal and $e$ was chosen to be the exponent of $\bigoplus_{i=1}^m C_i$.  Thus, $\iota(G)$ is contained in the kernel of $\mathrm{rot}$, which is not possible since $G$ is not left-orderable.
\end{proof}

With this in mind, we end with the following open question:

\begin{quest}
Does there exist a non-left-orderable group $G$ such that $G^n$ is circularly-orderable for all $n \in \NN$?
\end{quest}

\bibliographystyle{plain}
%\addcontentsline{toc}{section}{Bibliography}
\bibliography{rigidobstruction}
\end{document}